\documentclass[a4paper,12pt]{article}

\usepackage{amssymb}                                    
\usepackage{mathrsfs}                    
\usepackage{amsmath}                    
\usepackage{amsfonts}                   
\usepackage{amsthm}                     
\usepackage[english]{babel}                
\usepackage[arrow, matrix, curve]{xy}    
\usepackage{graphics}
\usepackage{hyperref}
\usepackage{enumerate}
\usepackage{color}
\usepackage{comment}

\newtheoremstyle{introduction}
  {3 pt}
  {4 pt}
  {\itshape}
  {}
  {\bfseries}
  {:}
  {.5em}
  {}

\theoremstyle{plain}
\newtheorem{thm}{Theorem}[section]
\newtheorem{corollary}[thm]{Corollary}
\newtheorem{proposition}[thm]{Proposition}
\newtheorem{lemma}[thm]{Lemma}
\newtheorem{example}[thm]{Example}
\newtheorem{prop}[thm]{Proposition}

\theoremstyle{introduction}
\newtheorem*{thm*}{Theorem} 

\theoremstyle{definition}
\newtheorem{definition}[thm]{Definition}
\newtheorem{remark}[thm]{Remark}

\DeclareMathOperator{\img}{Im}
\DeclareMathOperator{\Ker}{Ker}

\newcommand\Hom		{\mathrm{Hom}}

\newcommand\dSet	{\mathrm{dSet}}
\newcommand\Set     {\text{Set}}
\newcommand\sSet    {\text{sSet}}
\newcommand\Ab      {\mathrm{AbGr}}

\newcommand\Oper    {\text{Oper}}

\newcommand\iso     {\stackrel{\sim}{\longrightarrow}}

\newcommand{\calk }{\mathcal{K}}

\newcommand{\Z}{\mathbb{Z}}
\newcommand {\rt}{\mathrm{rt}}
\newcommand{\calK}{\mathcal{K}}
\newcommand{\calC}{\mathcal{C}}
\newcommand{\Sp}{\mathcal{S}\mathrm{p}}
\newcommand{\Mod}{\mathrm{Mod}}

\def\to{\rightarrow}

\def\sgn{\mathrm{sgn}}

\def\Aut{\mathrm{Aut}}
\def\Ch{\mathrm{Ch}}

\title{Homology of dendroidal sets}
\author{Matija Ba\v{s}i\'{c} and Thomas Nikolaus}

\begin{document}
\maketitle

\begin{abstract}
We define for every dendroidal set $X$ a chain complex and show that this assignment determines a left Quillen functor. Then we define the homology groups $H_n(X)$  as the homology groups of this chain complex. 
This generalizes the homology of simplicial sets. Our main result is that the homology of $X$ is isomorphic to the homology of the associated spectrum $\calK(X)$ as discussed in \cite{BN} and \cite{NikKtheory}. 
Since these homology groups are sometimes computable we can identify some spectra $\calk(X)$ which we could not identify before.   
\end{abstract}

\section{Introduction}

The definition of the singular homology of a topological space can be divided into several steps.  First, we consider the singular simplicial set, then we take the chain complex associated to this simplicial set and finally we compute the homology of this chain complex. The part we want to focus on for now is the construction which associates a chain complex $\Ch(S)$ to a simplicial set $S$. This chain complex is freely generated by the simplices of $S$ where an $n$-simplex has degree $n$ and the differential is given by the alternating sum of faces. 

The notion of a dendroidal set is a generalization of a simplicial set.
Dendroidal sets have been introduced by I. Moerdijk and I. Weiss (\cite{MoerW07}, \cite{MW09}) and shown to yield a good model for homotopy coherent operads by D.-C. Cisinski and I. Moerdijk (\cite{MC11a}, \cite{MC10}, \cite{MC11}). A dendroidal set $X$ has a set $X_T$ of $T$-dendrices for every tree $T$. We think of a tree as a generalization of a linearly ordered set. The category of simplicial sets embeds fully faithfully into the category of dendroidal sets $i_!: \sSet \to \dSet$, and in particular every simplicial set $S$ can be considered as a dendroidal set $i_!S$ which is `supported' on linear trees.  \\

The main construction of this paper, given in Section \ref{sectionbla}, is that of a non-negatively graded chain complex $\Ch(X)$ for every dendroidal set $X$ which generalizes the chain complex associated to simplicial sets.  More concretely, this means that for a simplicial set $S$ the two chain complexes $\Ch(S)$ and $\Ch(i_!S)$ are naturally isomorphic. 
  
The chain complex $\Ch(X)$ is, as a graded abelian group, generated by the isomorphism classes of (non-degenerate) dendrices of $X$. A $T$-dendrex has degree $|T|$, where $|T|$ is the number of vertices of $T$. 
The differential is, as in the simplical case, a signed sum of faces of the tree but the sign conventions are slightly more complicated than the simplicial signs.
 We discuss a normalized and an unnormalized variant of this chain complex in parallel to the variants of the chain complex for simplicial sets. 
A related construction for planar dendroidal sets is discussed in the paper \cite{GLW}. 

We show that the functor $X \mapsto \Ch(X)$ is homotopically well-behaved. More precisely, it forms a left Quillen functor with respect to the \emph{stable} model structure on dendroidal sets (as introduced in \cite{BN}) and the projective (or injective) model structure on chain complexes. In particular the chain 
complex $\Ch(X)$ is an invariant of the stable homotopy type of the dendroidal set $X$. It follows that it can also be considered as an invariant of $\infty$-operads (modelled by the Cisinski-Moerdijk model structure on dendroidal sets).
We define the homology of a dendroidal set $X$ as the homology of the associated chain complex for a cofibrant replacement of $X$. It is a good invariant in the sense that it is computable in practice. For example we show in Corollaries \ref{cor_eins} and \ref{cor_zwei} that 
$$
H_n(\Omega[T]) = \begin{cases}
\Z^{\ell(T)} & \text{for } n = 0, \\
0 & \text{otherwise,}
\end{cases}
\qquad  \qquad
H_n(\Omega[T]/ \partial \Omega[T])
= \begin{cases}
\Z & \text{for } n = |T|,  \\
0 & \text{otherwise,}
\end{cases}
$$
where $\ell(T)$ is the number of leaves of $T$ and $|T|$ is the number of vertices of $T$. \\

In \cite{BN} and \cite{NikKtheory}, based on work of Heuts \cite{Heuts2}, we have shown that there is a functor which assigns to every dendroidal set $X$ a connective spectrum $\calK(X)$. 
This functor has the property that for a dendroidal set of the form $i_!Y$ it yields the suspension spectrum $\Sigma^\infty_+Y$ (Theorem 5.5. in \cite{NikKtheory}).
It also generalizes the  $K$-theory of symmetric monoidal categories. Since the definition is rather inexplicit it turns out to be hard to identify the spectrum associated to a dendroidal set even for `small' examples of dendroidal sets. 

The main result of the present paper is that for a dendroidal set $X$ the homology $H_n(X)$ agrees with the homology of the associated spectrum $\calK(X)$ (Theorem \ref{mainthm}). 
If we denote by $\mathrm{SHC}$ the stable homotopy category of spectra and by $\mathrm{grAb}$ the category of graded abelian groups, we may summarize the discussion by saying that there is a square 
$$\xymatrix{
\sSet \ar[rr]^{\Sigma^\infty_+}\ar[d]_{i_!} && \mathrm{SHC} \ar[d]^{H_*} \\
\dSet \ar[rru]^{\calK} \ar[rr]_{H_*} &  &\mathrm{grAb} 
}$$
which commutes up to natural isomorphism. 

Our main result allows us to learn something about the spectrum $\calK(X)$ associated to a dendroidal set $X$ without having to understand $\calK(X)$ first. For example we use this strategy to identify the spectrum associated to the dendroidal set $\Omega[T]/\partial \Omega[T]$ as the $n$-sphere $\Sigma^n \mathbb{S}$ where $n$ is the number of vertices of $T$. As one more application we show that the spectrum associated to the dendroidal version of the operad $A_\infty$ is trivial. 

\paragraph{Organization} In Section 2 we review the necessary background about dendroidal sets and prove some technical results that we use later. In Section 3 we introduce and discuss several conventions how to assign signs to faces and isomorphisms of dendroidal sets. These are used in the definition of the chain complexes $\Ch^{un}(X)$ and $\Ch(X)$ in Section 4. In Section 5 we prove that these chain complexes are homotopically well-behaved and equivalent. Our main result is contained in Section 6. Finally, section 7 contains a technical result that is needed to compute certain homologies.

\paragraph{Acknowledgements} We would like to thank Ieke Moerdijk for many helpful discussions and comments on the first draft of this paper.

\section{Preliminaries on dendroidal sets}

In this section we recall the basic notions concerning dendroidal sets. We use the definition of the category $\Omega$ of non empty finite rooted trees as in \cite{MoerW07} and \cite{MW09}.

A (non empty, finite, rooted) tree consists of finitely many edges and vertices. Each vertex has a (possibly empty) set of input edges and one output edge. 
For each tree $T$ there is a unique edge $\rt(T)$, called the root of $T$, which is not an input of any vertex. An edge which is not an output of a vertex is called a leaf. 
The set of leaves of a tree $T$ is denoted by $\ell(T)$. Edges which are both inputs and outputs  of vertices are called inner edges. 

Let $|T|$ be the number of vertices of $T$.  A vertex with no inputs is called a stump.  A vertex with one input is called unary and a tree with only unary vertices is called linear. A linear tree with $n$ vertices is denoted $L_n$. The trivial tree $L_0$ is a tree with no vertices and one edge. A tree with one vertex is called a corolla. A corolla with $n$ inputs is denoted $C_n$.   

Vertices of a tree $T$ generate a symmetric coloured operad $\Omega(T)$ and the morphisms in the category of trees $\Omega$ are given by morphisms of corresponding operads. Hence, $\Omega$ is a full subcategory of the category $\Oper$ of coloured operads. 

The category $\dSet$ of dendroidal sets is the category of presheaves of sets on the category of trees $\Omega$.  We denote by $\Omega[T]$ the dendroidal set represented by a tree $T$, i.e. $\Hom_\Omega(-, T)$. We denote $\eta:=\Omega[L_0]$. If $X$ is a dendroidal set and $T$ is a tree, we denote $X_T:=X(T)$. 
By the Yoneda lemma, we have $X_T=\Hom(\Omega[T], X)$. The elements of the set $X_T$ are called $T$-dendrices of $X$. If $f\colon S\to T$ is a morphism in $\Omega$ we denote $f^*=X(f)\colon X_T \to X_S$. 

Let $\Delta$ be the category of non-empty finite linear orders and order preserving functions. Then the category $\sSet$ of simplicial sets is the category of presheaves on $\Delta$. There is an inclusion $i\colon \Delta \to \Omega$ of categories given by $i([n])=L_n$. This inclusion induces a pair of adjoint functors 
\[
i_!: \xymatrix{\sSet \ar@<0.3ex>[r] & \dSet: i^* \ar@<0.7ex>[l]}.
\]

We also note that the inclusion $\Omega \to \Oper$ induces a pair of adjoint functors 
\[
\tau_d: \xymatrix{\dSet \ar@<0.3ex>[r] & \Oper: N_d \ar@<0.7ex>[l]}.
\]
We call $N_d$ the dendroidal nerve functor. 

Note that the inputs of a vertex of a tree are not ordered in any way. A planar structure on a tree $T$ consists of a linear order on the set of inputs  of each vertex. A planar tree is given by a tree with a planar structure. Each planar tree generates a non-symmetric operad. We let $\Omega_p$ be the category of planar trees thought of as a full subcategory of the category of non-symmetric coloured operads.  

The symmetrization functor from non-symmetric operads to symmetric operads restricts to a functor $\Sigma\colon  \Omega_p \to \Omega$ which on objects forgets the planar structure. The main distinction between planar and non-planar trees is that every automorphism in $\Omega_p$ is an identity, while in $\Omega$ there are non-trivial automorphism. 

There is a dendroidal set $P:\Omega^{op} \to \Set$ such that $P(T)$ is the set of planar structures of the tree $T$. We also say that $P$ is the presheaf of planar structures. 

As in the category $\Delta$, there are elementary face and degeneracy maps in $\Omega$ which generate all morphisms. Let $e$ be an edge of tree $T$ and let $\sigma_e T$ be the tree obtained from $T$ by adding a copy $e'$ of the edge $e$ and a unary vertex between the $e'$ and $e$.  There is an epimorphism $\sigma_e \colon \sigma_e T \to T$ in $\Omega$ sending the unary operation in $\Omega(\sigma_e T)(e'; e)$ to the identity operation in $\Omega(T)(e;e)$. We call a morphism of this type an elementary degeneracy map.  

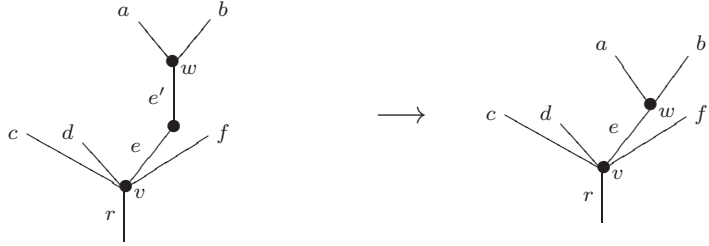
\[
  \xymatrix@R=4pt@C=10pt{
&&& _{a}&&_{b} &&& \\
&&& && &&& \\
& & && *=0{\,\,\, \bullet_w} \ar@{-}[uul]  \ar@{-}[uur] && \\
&&& && &&& \\
&  _c  & _d  &&*=0{\bullet }\ar@{-}[uu]^{e'}& _f & \\
&&& && &&& \\
&&&*=0{\,\,\, \bullet_v} \ar@{-}[uull] \ar@{-}[uul] \ar@{-}[uur]^e \ar@{-}[uurr]    &&&\\
&&& && &&& \\
&&&*=0{}\ar@{-}[uu]^r &&&
}
\xymatrix@R=4pt@C=10pt{ \\ \\ \\ \\ \longrightarrow \\ }
  \xymatrix@R=4pt@C=10pt{
  &&& && &&&& \\
&&& _{a}&&_{b} &&&& \\
&&& && &&&& \\
& _c & _d && *=0{\,\,\, \bullet_w} \ar@{-}[uul]  \ar@{-}[uur] &_f&& \\
&&& && &&&& \\
&&&*=0{\,\,\, \bullet_v} \ar@{-}[uull] \ar@{-}[uul] \ar@{-}[uur]^e \ar@{-}[uurr]    &&&&\\
&&& && &&&& \\
&&&*=0{}\ar@{-}[uu]^r &&&&
}
\]

If $e$ is an inner edge of a tree $T$, there is a tree $\partial_e T$ obtained by contracting the edge $e$ in $T$. The obvious monomorphism $\partial_e T\to T$ is called an inner elementary face map. 
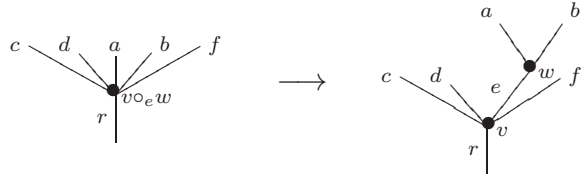
\[
\xymatrix@R=3pt@C=8pt{
&&&&&&\\
&_c & _d & _a &_b&_f& \\
&&&&&&\\
&&&*=0{\quad \,\,\, \bullet_{v \circ_e w}} \ar@{-}[uull] \ar@{-}[uul] \ar@{-}[uu] \ar@{-}[uur] \ar@{-}[uurr]   &&&\\
&&&&&&\\
&&&*=0{}\ar@{-}[uu]^r &&&
}
\xymatrix@R=3pt@C=8pt{ \\ \\ \\ \longrightarrow \\ }
  \xymatrix@R=3pt@C=8pt{
&&& _{a}&&_{b} &&&& \\
&&&&&&& \\
& _c & _d && *=0{\,\,\, \bullet_w} \ar@{-}[uul]  \ar@{-}[uur] &_f&& \\
&&&&&&& \\
&&&*=0{\,\,\, \bullet_v} \ar@{-}[uull] \ar@{-}[uul] \ar@{-}[uur]^e \ar@{-}[uurr]    &&&&\\
&&&&&&& \\
&&&*=0{}\ar@{-}[uu]^r &&&&
}
\]

Let $w$ be a top vertex of a tree $T$, i.e. let all inputs of $w$ be leaves. There is a tree $\partial_w T$ obtained by chopping off the vertex $w$ (and all its inputs) in $T$. The obvious monomorphism $\partial_w T \to T$ is called a top elementary face map. 

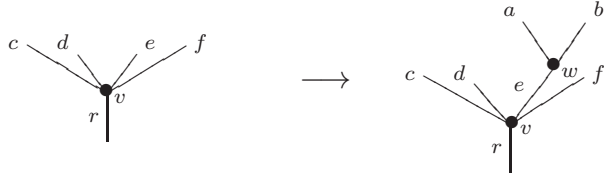
\[
  \xymatrix@R=3pt@C=8pt{
&&&&&&& \\
& _c & _d &&_e&_f \\
&&&&&&& \\
&&&*=0{\,\, \bullet_v} \ar@{-}[uull] \ar@{-}[uul] \ar@{-}[uur] \ar@{-}[uurr]   &&\\
&&&&&&& \\
&&&*=0{}\ar@{-}[uu]^r &&
}
\xymatrix@R=3pt@C=8pt{ \\ \\ \\ \longrightarrow \\ }
  \xymatrix@R=3pt@C=8pt{
&&& _{a}&&_{b} &&&& \\
&&&&&&& \\
& _c & _d && *=0{\,\,\, \bullet_w} \ar@{-}[uul]  \ar@{-}[uur] &_f&& \\
&&&&&&& \\
&&&*=0{\,\,\, \bullet_v} \ar@{-}[uull] \ar@{-}[uul] \ar@{-}[uur]^e \ar@{-}[uurr]    &&&&\\
&&&&&&& \\
&&&*=0{}\ar@{-}[uu]^r &&&&
}
\]

Let $v$ be a bottom vertex of a tree $T$, i.e. let the root $\rt(T)$ be the output of $v$, and let $e$ be an input of $v$ such that all other inputs of $v$ are leaves.  Note that such a pair $(v,e)$ does not exist for every tree.  If $T$ is a corolla, then $e$ is also a leaf and it can be any input of $v$, while for trees with more than one vertex $e$ is the unique inner edge attached to $v$ (if it exists at all). There is a tree $\partial_{v,e} T$ obtained by chopping off the vertex $v$ (with the root and all its inputs except $e$) in $T$. The obvious monomorphism $\partial_{v,e} T \to T$ is called a bottom elementary face map. 

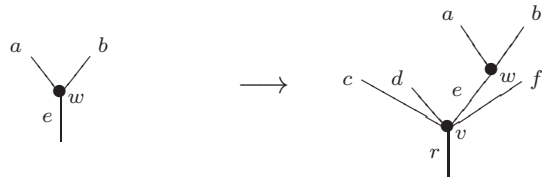
\[
  \xymatrix@R=3pt@C=8pt{
&& &&&&& \\
&& _a & &_b&  \\
&& &&&&& \\
&&&*=0{\,\,\, \bullet_w} \ar@{-}[uul]  \ar@{-}[uur]   &&\\
&& &&&&& \\
&&&*=0{}\ar@{-}[uu]^e &&
}
\xymatrix@R=3pt@C=8pt{ \\ \\ \\ \longrightarrow \\ }
  \xymatrix@R=3pt@C=8pt{
&&& _{a}&&_{b} &&&& \\
&& &&&&& \\
& _c & _d && *=0{\,\,\, \bullet_w} \ar@{-}[uul]  \ar@{-}[uur] &_f&& \\
&& &&&&& \\
&&&*=0{\,\, \bullet_v} \ar@{-}[uull] \ar@{-}[uul] \ar@{-}[uur]^e \ar@{-}[uurr]    &&&&\\
&& &&&&& \\
&&&*=0{}\ar@{-}[uu]^r &&&&
}
\]

We will say that $\partial_f \colon \partial_f T \to T$ is an elementary face map whenever $f$ is an inner edge $e$, top vertex $w$ or a pair $(v,e)$ of a bottom vertex $v$ with an input $e$ such that all other inputs are leaves. In that case we will say that $\partial_f T$ is an elementary face of $T$. 

Elementary face and degeneracy maps satisfy dendroidal identities. If $\partial_f T$ and $\partial_g \partial_f T$ are elementary faces of $T$ and $\partial_f T$ respectively, there are also elementary faces $\partial_{g'}T$ (with $g' \neq f$) and $\partial_{f'} \partial_{g'} T$ and the dendroidal identity  $\partial_g \partial_f  = \partial_{f'} \partial_{g'} $ is satisfied. 

Note that in most cases $g'=g$ and $f'=f$ (think of two inner edges $f$ and $g$), but this is not always the case (think of $f$ and $g$ being an inner edge and a top vertex attached to it, while $g'$ and $f'$ are both top vertices). Nonetheless, the above dendroidal identity is sufficiently good for our purposes. 

Moreover, there are obvious dendroidal identities relating two elementary degeneracy maps, an elementary degeneracy map with an elementary face map or any elementary map with an isomorphism. We refer the reader to Section 2.2.3 in \cite{MoerBar} for more details.

\begin{lemma}[\cite{MoerBar}, Lemma 2.3.2]
Every morphism in $\Omega$ can be factored in a unique way as a composition of elementary face maps followed by an isomorphism and followed by a composition of  elementary degeneracy maps. 
\end{lemma}

\begin{definition}
A dendrex is called \emph{degenerate} if it is in an image of $\sigma_e^*$, where $\sigma_e$ is an elementary degeneracy map. A dendrex which is not degenerate is called \emph{non-degenerate}.
\end{definition}

\begin{lemma}[\cite{MoerBar}, Lemma 3.4.1] \label{fact_nondeg}
Let $X$ be a dendroidal set and $x\in X_T$ a dendrex of $X$, for some tree $T$. There is a unique composition of elementary degeneracy maps $\sigma\colon T\to S$ and a unique non-degenerate dendrex $x^{\#}\in X_S$ such that $x=\sigma^*(x^{\#})$. 
\end{lemma}

\begin{definition}
Any elementary face map $\partial_f : \partial_f T \to T$ induces a map of representable dendroidal sets $\partial_f : \Omega[\partial_f T] \to \Omega[T]$. The union of all images of maps $\partial_f : \Omega[\partial_f T] \to \Omega[T]$ is denoted by $\partial \Omega[T]$. The inclusion $\partial \Omega[T] \to \Omega[T]$ is called a \emph{boundary inclusion}.
\end{definition}

\begin{definition}
A monomorphism $f\colon A\to B$ of dendroidal sets is called \emph{normal} if the action of the automorphism group $\Aut(T)$ on $B_T\setminus f(A_T)$ is free, for every tree $T$. 
We say that a dendroidal set $A$ is \emph{normal} if $\emptyset \to A$ is a normal monomorphism. 
\end{definition}

\begin{prop}[\cite{MoerBar}, Proposition 3.4.4] \label{boundary normal}
The class of all normal monomorphisms is the smallest class of monomorphisms closed under pushouts and transfinite compositions that contains all boundary inclusions $\partial\Omega[T] \to \Omega[T]$. 
\end{prop}

Using the definition of a normal monomorphism it is easy to see that if $f \colon A\to B$ is any morphism of dendroidal sets and $B$ is normal, then $A$ is also normal. If $f$ is a monomorphism and $B$ is normal, then $f$ is a normal monomorphism. 

\begin{definition}
For an elementary face map $\partial_f : \partial_f T \to T$ we denote by $\Lambda^f[T]$ the union of images of all elementary face maps $\partial_g : \Omega[\partial_g T] \to \Omega[T], g\neq f$. 

The inclusion $\Lambda^f[T] \to \Omega[T]$ is called a \emph{horn inclusion}. A horn inclusion is called \emph{inner} (respectively, \emph{top} or \emph{bottom}) if $\partial_f$ is an inner (resp. top or bottom) elementary face map. 
\end{definition}

\begin{definition}[\cite{BN}] \label{deffullyKan}
A dendroidal set $X$ is called \emph{fully Kan} if the induced map
\[
\Hom(\Omega[T], X) \to \Hom(\Lambda^f[T], X)
\]
is a surjection for every horn inclusion $\Lambda^f [T] \to \Omega[T]$.
\end{definition} 

\begin{thm}[\cite{BN}] \label{thm_equivalence}
There is a combinatorial left proper model structure on dendroidal sets, called the stable model structure,  for which the cofibrations are normal monomorphisms and fibrant objects are fully Kan dendroidal sets. 

The stable model structure is Quillen equivalent to the group-completion model structure on $E_\infty$-spaces. In particular, the underlying $\infty$-category of fully Kan dendroidal sets is equivalent to the $\infty$-category of connective spectra. 
\end{thm}

This theorem in particular implies that for every dendroidal set $D$ there is an associated connective spectrum which we denote by $\calK(D)$. The assignment $D \mapsto \calK(D)$ has been investigated in \cite{NikKtheory}. One of the 
main results there is that this functor generalizes $K$-theory of symmetric monoidal categories.

\begin{remark}
By Proposition \ref{boundary normal}, the generating cofibrations are given by boundary inclusions $\partial \Omega[T]\to \Omega[T]$. Horn inclusions $\Lambda^f[T]\to \Omega[T]$ are trivial cofibrations, 
but it is not known whether  the set of horn inclusions is the set of generating trivial cofibrations. On the other hand, fibrant objects and fibrations between fibrant objects are characterized by the right lifting property
 with respect to all horn inclusions, cf. Theorem 4.6 in \cite{BN} and Proposition 5.4.3 and Proposition 5.4.5 in \cite{basic-thesis} . 
\end{remark}

\begin{lemma}\label{sep18}
Let $M$ be a model category and $F: \dSet \to M$ a left adjoint functor. Then $F$ is a left Quillen functor with respect to the stable model structure  if and only if $F$ sends
boundary inclusions to cofibrations and horn inclusions to trivial cofibrations in $M$.  
\end{lemma}

\begin{proof}
If $F$ is left Quillen, then $F$ clearly sends
boundary inclusions to cofibrations and horn inclusions to trivial cofibrations in $M$. 

To prove the converse, let us assume that $F$ sends boundary inclusions to cofibrations and horn inclusions to trivial cofibrations in $M$. Since cofibrations in the stable model structure are generated as a saturated class by boundary inclusions, it follows that $F$ preserves cofibrations. 
Let $G$ be the right adjoint of $F$. It is a well-known fact about model categories that trivial cofibrations are characterized by the lifting property against fibrations between fibrant objects, see e.g. \cite[Lemma A.6.1]{JoyalTierney}. By this fact and adjunction it follows that $F$ preserves trivial cofibrations if $G$ preserves fibrations between fibrant objects.
We now use that fibrant objects and trivial fibrations between fibrant objects in $\dSet$  can be characterized by the lifting property against horn inclusions. Thus  another application of the adjunction property proves the claim. \end{proof}

\begin{lemma} \label{LemmaEq}
Let $M$ be a model category and let $F, G \colon \dSet \to M$ be left adjoint functors that send normal monomorphisms to cofibrations. If there is a natural transformation $\alpha \colon F \to G$ such that $\alpha_{\Omega[T]} \colon F(\Omega[T]) \to G(\Omega[T])$ is a weak equivalence for every tree $T$, then $\alpha_X\colon F(X) \to G(X)$ is a weak equivalence for every normal dendroidal set $X$. 
\end{lemma}

\begin{proof}
For a non-negative integer $n$, we say that a dendroidal set $X$ is $n$-\emph{dimensional} if $X$ has no non-degenerate dendrex of shape $T$ for $|T|>n$. 

We prove by induction on $n$ that if $X$ is a normal $n$-dimensional dendroidal set, then $F(X) \to G(X)$ is a weak equivalence. If $X$ is $0$-dimensional, then $X$ is just a coproduct of copies of $\eta$. By the assumption, $F(\eta) \to G(\eta)$ is a weak equivalence, so $F(X) \to G(X)$ is a weak equivalence since it is a coproduct of weak equivalences between cofibrant objects. 

For the inductive step, assume $X$ is an $n$-dimensional normal dendroidal set and let $X'$ be its $(n-1)$-skeleton. Then $X = X' \cup_{\coprod \partial\Omega[T]} \coprod \Omega[T]$, where the coproduct varies over all isomorphism classes of non-degenerate dendrices in $X_T$ with $|T|=n$.

Since $F$ and $G$ are left adjoints they preserve colimits, so there is a commutative diagram 
\[
\xymatrix{ \coprod F(\partial\Omega[T]) \ar[rr] \ar[rd]^\sim \ar[dd] &  & F(X')  \ar[dd] \ar[rd]^\sim & \\ 
 &\coprod G(\partial\Omega[T])  \ar[rr] \ar[dd] & & G(X') \ar[dd] \\
\coprod F(\Omega[T])  \ar[rr] \ar[rd]^\sim & & F(X)  \ar[rd]& \\
& \coprod G(\Omega[T]) \ar[rr] & &  G(X)
}
\]
where all the objects are cofibrant, the back and front sides are pushout squares and the two vertical maps on the left are cofibrations. The two maps in the upper square are weak equivalences by the inductive hypothesis. The map $\coprod F(\Omega[T]) \to \coprod G(\Omega[T])$ is a weak equivalence by the assumption. Hence $F(X) \to G(X)$ is also a weak equivalence. 

Finally, for a normal dendroidal set $X$, consider the skeletal filtration of $X$: 
\[
\emptyset=X^{(-1)}\subseteq X^{(0)} \subseteq X^{(1)} \to \ldots \to X^{(n)} \subseteq \ldots  
\]

Since $X^{(n)}$ is $n$-dimensional, we have shown already that $F(X^{(n)}) \to G(X^{(n)})$ is a weak equivalence between cofibrant objects. Since $F$ (resp. $G$) preserves colimits, $F(X)$ (resp. $G(X)$) is a filtered colimit of $F(X^{(n)})$ (resp. $G(X^{(n)})$) and hence $F(X) \to G(X)$ is a weak equivalence, too. 
\end{proof}

\begin{corollary}\label{15sep2}
Let $F,G: \dSet \to M$ be two left Quillen functors and $\varphi: F \Rightarrow G$ a natural transformation such that $\varphi_\eta: F(\eta) \to G(\eta)$ is an equivalence. Then 
$\varphi_X: F(X) \to G(X)$ is an equivalence for any normal dendroidal set $X$. In particular, $F$ and $G$ induce equivalent functors on homotopy categories. 
\end{corollary}

\begin{proof} For any tree $T$, the inclusion of leaves 
\[ \bigsqcup_{\ell(T)} \eta\to \Omega[T]\]
is a stable trivial cofibration.
Since $F$ and $G$ are left Quillen, $F(\Omega[T])\to G(\Omega[T])$ is a weak equivalence, too. The result follows from \ref{LemmaEq}.
\end{proof} 

The last corollary gives an easy criterion to check that two left Quillen functors are equivalent once we are given a natural transformation between the two. We will need later a stronger version of that result 
where we can drop the assumption that we are already given a natural transformation. To prove this stronger result we have to rely on results of \cite{GGN} which are obtained in the setting of $\infty$-categories. 
Thus we will also state the result in the setting of $\infty$-categories. But note that a left Quillen functor between model categories gives rise to a left adjoint functor between $\infty$-categories.	

\begin{proposition}\label{prop_cool}
Let $F,G: \dSet_\infty \to \calC$ be two left adjoint functors of $\infty$-categories, where $\dSet_\infty$ denotes the $\infty$-category underlying the stable model structure on dendroidal sets.
Assume that $\calC$ is presentable and additive. The latter means that the homotopy category $\mathrm{Ho}(\calC)$ is additive. If $F(\eta) \simeq G(\eta)$ in $\calC$ then the functors $F$ and $G$ are equivalent. In particular, for every dendroidal set 
$X$ there is an equivalence $F(X) \simeq G(X)$ in $\calC$.
\end{proposition}
\begin{proof}
We first use that the $\infty$-category $\dSet_\infty$ is equivalent to the $\infty$-category $\Sp^{\geq 0}$ of connective spectra as a result of the equivalence mentioned in Theorem \ref{thm_equivalence}. It is shown in \cite{BN} that under this equivalence the dendroidal set $\eta$ is sent to the sphere. 
Then we use Corollary 4.9. in \cite{GGN}, which states that connective spectra form the `free additive' $\infty$-category on one generator (which is the sphere). This implies that two left adjoint $\infty$-functors $F,G\colon \Sp^{\geq 0} \to \calC$ from the $\infty$-category of connective spectra to an additive $\infty$-category $\calC$ are equivalent if they coincide on the sphere. Thus $\dSet_\infty$ satisfies the same universal property which proves the statement.
\end{proof}

\begin{corollary}
Let $F,G: \dSet \to M$ be two left Quillen functors where $M$ is an additive combinatorial model category. If there exists an equivalence $F(\eta) \to G(\eta)$ in $\mathrm{Ho}(M)$ then  there is a zig-zag of natural equivalences between $F$ and $G$. In particular, the induced functors $F,G: \mathrm{Ho}(\dSet) \to \mathrm{Ho}(M)$ on homotopy categories are equivalent.
\end{corollary}

\section{Some conventions about signs}

In this section we describe a labelling of vertices of a planar tree and two sign conventions for faces and automorphisms. These labels and signs will be used in the definition of the homology of a dendroidal set. One of these two sign conventions is taken from \cite[Section 4.5]{GLW}. 

Recall that planar trees are trees with extra structure - the set of inputs of each vertex carries a total order. We depict planar trees by drawing the inputs from the left to the right in increasing order. There is a dendroidal set $P:\Omega^{op} \to \Set$ such that $P(T)$ is the set of planar structures of the tree $T$. We call it also the presheaf of planar structures. Note that $P=A_{\infty}=N_d(Ass)$ is the dendroidal nerve of the operad for associative algebras and it is a normal dendroidal set. 

Let $(T,p)$ be a planar tree, i.e. $T$ is a non-planar tree and $p$ a planar structure on $T$. For every face map $f\colon S \to T$ there is a planar structure on $S$ given as $P(f)(p)$, so that $f$ is a map of planar trees with these planar structures. 

We define a labelling of the vertices of a planar tree with $n$ vertices by the numbers $0$,$1$,\ldots, $n-1$ as follows. We label the vertex above the root edge with 0 and then proceed recursively. 
Whenever we label a vertex we continue labelling the vertices of its leftmost branch (until we reach a top vertex), then we label the vertices of the second branch from the left and so on until we have labelled
 all the vertices. An example of such a labelling is given below.

\begin{definition}
We assign a sign $\sgn_p(\partial_a)\in\{-1,+1\}$ to each elementary face map  $\partial_a: \partial_a T \to T$ using the labelling of the planar tree $(T,p)$ as follows:
If $T$ is a corolla, we assign $-1$ to the inclusion of the root edge 
and $+1$ to the inclusion of a leaf. If $T$ has at least two vertices, we assign $+1$ to the root face, which is the face obtained by chopping off the root vertex (and which only exists if the root vertex has only one inner edge assigned to it).
We assign $(-1)^k$ to the the face $\partial_e T \to T$ if $e$ is an inner edge which is attached to vertices labelled with $k$ and $k-1$ and we assign $(-1)^{k+1}$ to $\partial_v$ if $v$ is a top vertex labelled with $k$.
\end{definition}

Here is an example of such a labelling of the vertices. The signs associated to the inner faces are shown next to the corresponding inner edge and the signs associated to the top faces are shown next to one of the leaves. 
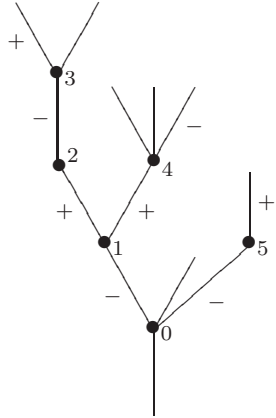
\[
\xymatrix@R=10pt@C=12pt{
&&&&&&&\\
&&&&&&&\\
&& *=0{\, \, \bullet_3} \ar@{-}[uul]^{+}  \ar@{-}[uur] &&&&& \\
&&&&&&&\\
&& *=0{\,\,\, \bullet^2} \ar@{-}[uu]^{-} && *=0{\bullet} \ar@{-}[uul]\ar@{-}[uu]\ar@{-}[uur]_{-}\ar@{-}_{4}&&& \\
&&&&&&&\\
&&& *=0{\, \, \bullet_{1} }\ar@{-}[uul]^{+} \ar@{-}[uur]_{+}&&& *=0{\, \, \bullet_5} \ar@{-}[uu]_{+} &\\
&&&&&&&\\
&&&& *=0{\, \, \bullet_0} \ar@{-}[uul]^{-} \ar@{-}[uur] \ar@{-}[uurr]_{-} &&& \\
&&&&&&&\\
&&&&*=0{}\ar@{-}[uu] &&&
}
\]

Next we define a sign convention that will be used when we consider different planar structures. 

\begin{definition}
Let $T$ be a tree and let $p', p\in P_T$ be two planar structures. Each of these planar structures gives a labelling of the vertices of $T$ as described above. Thus there is a permutation on the set of labels $\{1,...,n-1\}$ which sends the labelling induced by $p'$ to the labelling induced by $p$ (we omit the label 0 since the root vertex must be fixed). We define $\sgn(p', p) \in\{-1,+1\}$ to be the sign of that permutation. 
\end{definition}

\begin{example}
Here is a simple example. Let $(T,p)$ be the following planar tree
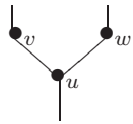
\[
\xymatrix@R=10pt@C=12pt{
&&&&&&&\\
&&& *=0{\, \, \,\, \bullet_{v} }\ar@{-}[u] && *=0{\, \, \, \bullet_w} \ar@{-}[u] &&\\
&&&& *=0{\, \, \bullet_u} \ar@{-}[ul] \ar@{-}[ur]  &&& \\
&&&&*=0{}\ar@{-}[u] &&&
}
\]
The same tree $T$ has one more planar structure $p'$. The vertices $v$ and $w$, respectively, have labels $1$ and $2$ in $p$, but labels $2$ and $1$ in $p'$, so $\sgn(p',p)=-1$. 
\end{example}

Let $\partial_e : \partial_e T \to T$ be an elementary face map. If $p\in P_T$ is a planar structure on $T$, then we denote $p_e=P(\partial_e)(p)\in P_{\partial_e T}$. 
 
\begin{lemma} \label{signs}
Let $\partial_e: \partial_e T \to T$ be an elementary face map. For any two planar structures $p', p\in P_T$ we have  
\begin{equation*}
\sgn(p', p) \cdot \sgn_p(\partial_e) = \sgn (p'_e, p_e) \cdot \sgn_{p'}(\partial_{e}). 
\end{equation*}
\end{lemma}

\begin{proof}
If $T$ is a corolla, the statement is true since all the terms are $+1$. 
Let $|T|=n+1$ be the number of vertices $T$ and $\tau\in \Sigma_n$ the permutation assigned to  the planar structures $p'$ and $p$. Suppose first that $e$ is an inner edge of $T$. Let $k$ be the label given to the vertex above $e$ and $\tau(k)=l$. Then $\sgn_p(\partial_e)=(-1)^k$ and $\sgn_{p'}(\partial_e)=(-1)^l$.

We denote by $\tau_e\in \Sigma_{n-1}$ the  permutation assigned to the planar structures $p'_e$ and $p_e$. Observe that the permutation $\tau_e:\{1,2,...,n-1\} \to \{1,2,...,n-1\}$ is obtained from the permutation $\tau:\{1,2,...,n\} \to \{1,2,...,n\}$ in the following way. We delete $k$ in the domain of $\tau$ and relabel the elements greater than $k$ by decreasing them by 1. Also we delete $l$ in the codomain of $\tau$ and relabel the elements greater than $l$ by decreasing them by 1. Now we compare the number of inversions of $\tau$ (i.e. the instances of pairs $(a,b)$ such that $a,b\in \{1,2,...,n\}$,  $a<b$ and $\tau(a)>\tau(b)$) to the number of 
inversions of $\tau_e$. Actually the inversions in $\tau_e$ are in bijection with the inversions $(a,b)$ of $\tau$ such that $a$ and $b$ are different than $k$ 
(if $a,b\neq k$ then the mentioned relabelling does not affect the relative order of $\tau(a)$ and $\tau(b)$ when considered in the codomain of $\tau_e$). 
So we need to calculate the number of elements of the set 
\begin{equation} \label{perm}
 \{(a,k) : 1\leq a < k, \tau(a) > l \} \cup  \{(k,b) : k< b \leq n, \tau(b) < l\}. 
\end{equation}
Denote by $p$ the number of elements of the set $\{a : 1\leq a < k, \tau(a) > l \}$. Then the number of elements of the set $\{a : 1\leq a < k, \tau(a) < l \}$ is $k-p-1$.
But the elements of the latter set are in bijection with the elements of $\{c : 1\leq c <l, \tau^{-1}(c) <k\}$. This implies that the number of elements of the set 
$\{c : 1\leq c <l, \tau^{-1}(c) > k\}$ is $l-(k-p-1)-1=l-k+p$, and this set is in bijection with $\{b : k< b \leq n, \tau(b) < l\}$. 
So the number of the elements of the set \ref{perm} is $l-k+p+p=l-k+2p$ and we conclude $\sgn(\tau)=(-1)^{l-k+2p}\sgn(\tau_e)$. 

If we suppose $\partial_e$ is a face map corresponding to a top vertex of $T'$ labelled by $k$ and $\tau(k)=l$, then in the same way we conclude $\sgn(\tau)=(-1)^{l-k+2p}\sgn(\tau_e)$. 
Since in this case $\sgn_p(\partial_e)=(-1)^{k+1}$ and $\sgn_{p'}(\partial_e)=(-1)^{l+1}$, the statement of the lemma holds.

If $\partial_e$ is a face map corresponding to a root vertex, then $\sgn(\tau)=\sgn(\tau_e)$ (because in this case $\tau(1)=1$ and $\tau_e$ is obtained by deleting 1 in domain and codomain of $\tau$) and $\sgn_p(\partial_e)=\sgn_{p'}(\partial_e)=1$ by definition. 
\end{proof}

\section{The chain complex of a dendroidal set}\label{sectionbla}

In this section we define two chain complexes associated to a dendroidal set such that the definitions extend the construction of the normalized and unnormalized chain complex associated to a simplicial set. \\

Recall that for a tree $T$ we denote by $|T|$ the number of vertices of $T$. Let $X$ be a dendroidal set and $n\in \mathbb{N}_0$. We consider the free abelian group
\begin{equation} \label{def_c}
C(X)_n :=\bigoplus_{T\in \Omega,  |T|=n} \bigoplus_{p\in P_T} \mathbb{Z}\langle X_T \rangle 
\end{equation}
generated by triples $(T,p,x)$ where $(T,p)$ is a planar tree and $x \in X_T$.
For trees $T$ and $T'$, planar structures $p\in P_T$ and $p'\in P_{T'}$, an isomorphism $\tau: T'\to T$ and a dendrex $x\in X_T$  we consider the free subgroup $A(X)_n$ generated by 
\begin{equation} \label{relation}
(T,p,x) - \sgn(p', \tau^* p) (T', p', \tau^*(x)). 
\end{equation}
Here  $\tau^*(x)$ is  $X(\tau)(x)$ for $X(\tau) \colon X_T \to X_{T'}$ and  $\tau^*(p)$ denotes $P(\tau)(p)$ for $P(\tau)\colon P_T \to P_{T'}$. 

\begin{definition} Let $X$ be a dendroidal set. For each $n\in \mathbb{N}_0$ we define an abelian group $\Ch^{un}(X)_n$ as the quotient 
\begin{equation*}
\Ch^{un}(X)_n:=  C(X)_n / A(X)_n 
\end{equation*}
or more suggestively 
\begin{equation*} 
\Ch^{un}(X)_n:= \frac{\left( \bigoplus_{T\in \Omega, |T|=n} \bigoplus_{p\in P_T} \mathbb{Z}\langle X_T \rangle \right)} { (T,p,x) \sim \sgn(p', \tau^* p) (T', p', \tau^*(x)) }.
\end{equation*}
\end{definition}

Note that $\Ch^{un}(X)_n$ is a free abelian group since we have identified generators of a free abelian group $C(X)_n$. The generators of $\Ch^{un}(X)_n$ are in bijection with the isomorphism classes of dendrices of $X$.
 Each representative carries additional information: a planar structure, which is used only for the definition of the differential that we will give now. As we will show, it does not matter which planar structure we use. We write $[T,p,x]$ for the generator represented by the triple $(T,p,x)$.

\begin{definition}
Let $X$ be a dendroidal set. For every positive integer $n$, we define a map $d\colon \Ch^{un}(X)_n\to \Ch^{un}(X)_{n-1}$ on generators by
\begin{equation*}
d([T,p,x]):=\sum_{\partial_e \colon \partial_e  T \to T} \sgn_p(\partial_e) [\partial_e  T, p_e, \partial_e^* x], 
\end{equation*}
and extend it additively. The sum is taken over the set of elementary face maps of $T$.
\end{definition}

Note that by the definition of elementary face maps, its colours are subsets of the colours of $T$. There can be other monomorphisms $S \to T$ which are isomorphic over $T$ to such elementary face maps. 
These are not included in the indexing set of our sum. One could also sum over isomorphism classes of such monomorphisms but that leads to complications in terms of signs.

\begin{lemma} \label{well-def}
The map $d\colon \Ch^{un}(X)_n \to \Ch^{un}(X)_{n-1}$ is well-defined. 
\end{lemma}

\begin{proof}
Let $x\in X_T$ and $x'=\tau^* x\in X_{T'}$ for some isomorphism $\tau \colon T' \to T$. If $p\in P_T$ and $p'\in P_{T'}$ are two planar structures, we have $[T,p,x]=\sgn(p', \tau^*p)[T', p', \tau^*x]$. So, we need to prove that  
\[
\sum_{\partial_e \colon \partial_e  T' \to T'} \sgn_{p'}(\partial_e) [\partial_e  T', p'_e, \partial_e^* (x')] =\sgn(p',\tau^* p) \sum_{\partial_f \colon \partial_f  T \to T} \sgn_{p}(\partial_f) [\partial_f  T, p_f, \partial_f^*(x)],
\]
where the sums are taken over the set of elementary face maps of $T'$ and $T$, respectively. 

\noindent There is a unique isomorphism $\tau_e \colon \partial_e  T' \to \partial_{\tau(e)} T$ such that $\tau \partial_e = \partial_{\tau(e)} \tau_e$. Note that \[(\tau^*p)_e=P(\partial_{e})P(\tau)(p)=P(\tau_e)P(\partial_{\tau(e)})(p)=\tau_e^* p_{\tau(e)}.\]

\noindent Hence, lemma \ref{signs} implies 
\begin{align*}
\sgn_{p'}(\partial_e)[\partial_e  T', p'_e, \partial_e^* x']&=\sgn_{p'}(\partial_{e})[\partial_e  T', p'_e, \partial_e^* \tau^* x] \\
&=\sgn_{p'}(\partial_{e})[\partial_e T', p'_e, \tau^*_e \partial_{\tau(e)}^* x] \\
&=\sgn_{p'}(\partial_{e})\sgn(p'_e, \tau_e^* p_{\tau(e)}) [\partial_{\tau(e)} T, p_{\tau(e)}, \partial^*_{\tau(e)} x]  \\
&= \sgn(p', \tau^* p) \sgn_{p}(\partial_{\tau(e)}) [\partial_{\tau(e)} T, p_{\tau(e)}, \partial^*_{\tau(e)} x] \\
\end{align*}
The set of elementary face maps $\partial_e \colon \partial_e T' \to T'$ is in bijection with the set of elementary face maps $\partial_f\colon \partial_f T \to T$ by $e \mapsto f=\tau(e)$, so collecting these terms gives the desired statement.
\end{proof}

\begin{proposition}
The graded abelian group $(Ch^{un}(X), d)$ is a chain complex. 
\end{proposition}

\begin{proof}
We need to prove that $d^2=0$. Consider $x\in X_T$ and a planar structure $p$. We write $[x]$ instead of $[T,p,x]$ as the planar structure is clear from the context. We have the following calculation
\begin{align*}
 d^2([x]) &= d\left( \sum_{\partial_e^* \colon \partial_e T \to T} \sgn_p(\partial_e) [\partial_e^* x] \right)\\
	  &= \sum_{\partial_e \colon \partial_e  T \to T}  \sum_{\partial_f \colon \partial_f \partial_e  T \to \partial_e T} \sgn_p(\partial_e) \sgn_{p_e}(\partial_f) [\partial_f^* \partial_e^* x]
\end{align*}

For every two elementary face maps $\partial_e\colon \partial_e T \to T$ and $\partial_f \colon \partial_f \partial_e T \to \partial_e T$ there are elementary face maps $\partial_{f'} \colon \partial_{f'} T \to T$ and $\partial_{e'}\colon \partial_{e' } \partial_{f'} T \to \partial_{f'} T$ such that $\partial_e \partial_f = \partial_{f'} \partial_{e'}$.
The sign convention for faces of a planar tree is defined exactly so that the following holds
\begin{equation*}
\sgn_p(\partial_e) \sgn_{p_e}(\partial_f) = -\sgn_p(\partial_{f'}) \sgn_{p_{f'}}(\partial_{e'}).
\end{equation*}

This follows easily from the sign convention by inspection and it is also stated in \cite{GLW} as Lemma 4.3. 
Hence every term in the above sum appears exactly twice, each time with a different sign. 
This proves that the above sum is zero, i.e. $d^2=0$. 
\end{proof}

\noindent Finally, for a morphism $f:X\to Y$ of dendroidal sets, we define 
\begin{equation*}
\Ch^{un}(f)_n([T,p,x])=[T,p,f(x)], \quad x\in X_T. 
\end{equation*}
Since $f$ is a morphism of dendroidal sets it follows that $\Ch^{un}(f)_n$ is a well-defined morphism of chain complexes. 
In this way we obtain a functor $\Ch^{un}\colon \dSet \to \Ch_{\geq0}$. 

\begin{proposition}
Let $X$ be a dendroidal set. Consider the subgroups $D(X)_n \subset Ch^{un}(X)_n$ generated by the classes of degenerate dendrices. Then $D(X)$ is a subcomplex, i.e. it is closed under taking differentials.
\end{proposition}

\begin{proof}
We need to check that the differential restricts to classes represented by degeneracies. Let $\sigma: \sigma T \to T $ be a degeneracy map. Then the tree $T$ has 
two adjacent face maps $\partial_f$ and $\partial_{f'}$ which are equal up to an isomorphism of $\partial_f  T$ and $\partial_{f'} T$. 
Let $x\in X_T$. Then 
\begin{equation*}
 d[\sigma^* x]= \sum_{\partial_e \colon \partial_e  T \to T} \sgn_d(\partial_e) [\partial_e^* \sigma^* x].
\end{equation*}

By the dendroidal identities $\sigma$ commutes with all face maps $\partial_e$ except $\partial_f$ and $\partial_{f'}$, but  
$\sgn_d(\partial_f)[\partial_f^* \sigma^* x] =-\sgn_d(\partial_{f'}) [\partial_{f'}^* \sigma^* x]$. We conclude that the above sum is equal to the sum of classes represented 
by degeneracies.  
\end{proof}

\begin{lemma} \label{AcycDeg}
Let $X$ be a dendroidal sets such that for every non-degenerate dendrex $x\in X_T$ the associated map $x\colon \Omega[T]\to X$ is a monomorphism. Then the subcomplex $D(X)$ is acyclic. 
\end{lemma}

\begin{proof}
Let us fix a linear order on the set $X_{\eta}$. 
We will first establish some terminology. If $x\in X_T$ is a dendrex and $e$ is an edge of $T$, then we say that $e^*(x)\in X_\eta$ is a colour of $x$.

Let $x\in X_T$ be a degenerate dendrex of shape $T$. By Lemma \ref{fact_nondeg}, there is a composition of elementary degeneracy maps 
$\sigma\colon T\to S$ and a non-degenerate dendrex $x^{\#}\in X_S$ such that $x=\sigma^*(x^{\#})$.
If $e$ is an edge of $S$ such that $\sigma$ factors through the elementary degeneracy map $\sigma_e$, then the preimage $\sigma^{-1}(e)$ in $T$ has at least two elements. We think of all the elements in $\sigma^{-1}(e)$ as copies of $e$ as, obviously, $x\colon \Omega[T]\to X$ maps all the elements in $\sigma^{-1}(e)$ to the same element $e^*(x)\in X_\eta$.  If we denote $a=e^{*}(x)$ and $\sigma^{-1}(e)$ has $k$ elements, we say that $x$ factors through a $k$-fold degeneracy on $a$.  
Let us consider all elements $a\in X_\eta$ such that $x$ factors through a $k$-fold degeneracy on $a$ for some $k$. The smallest such $a$ with respect to our fixed order on $X_\eta$ will be called the smallest 
degenerate colour of $x$.

If $a$ is the smallest degenerate colour of $x$ and $x$ factors through a $k$-fold degeneracy on $a$, we say $x$ is canonical if $k$ is an odd number.
Consider a class in $D(X)_n$ and its two representatives $x$ and $y$. We have that $x$ is canonical if and only if $y$ is canonical. So it is well-defined to say that a class in $D(X)_n$ is canonical if any of its representatives is canonical. We define $A_n$ as the set of all canonical generators of $D(X)_n$ and $B_n$ as the set of all generators of $D(X)_n$ that are not canonical. 
A bijection between $B_n$ and $A_{n+1}$ is established by degenerating $x$ at the smallest degenerate colour $a$. Note that $A_0, B_0$ and $A_1$ are empty sets and $d(x)=0$ for all $x\in B_1$. 
Let $C_n=D(X)_n$ and $C_{n, can}= Z\langle A_n\rangle$. If we define $w:C_n \to \mathbb{N}_0$ to be 
\begin{equation*}
w(x)=\left\{ 
\begin{array}{ll} 0, & \textrm{if } x\in C_{n,can} \\ 1, & \textrm{otherwise} \end{array} 
\right.
\end{equation*}
then all assumptions of Proposition \ref{acyclicity} obviously hold. So all homology groups of $D(X)$ vanish. 
\end{proof}

\begin{definition}
We define the normalized chain complex as the quotient 
\begin{equation*}
\Ch(X)_\bullet: =\Ch^{un}(X)_\bullet/ D(X)_\bullet
\end{equation*}
\end{definition}

\begin{remark}
Since every dendrex is either degenerate or non-degenerate, we can identify the quotient $\Ch(X)_n$ as a subgroup 
of $\Ch^{un}(X)$ generated by all classes of non-degenerated dendrices. This inclusion however is not compatible with the differentials. The reason is that the differential of non-degenerated 
dendrices is not necessarily a linear combination of non-degenerated dendrices as the following example shows.

Let $x$ be a dendrex of some dendroidal set $X$ of the following shape  
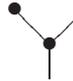
\[
\xymatrix@R=7pt@C=7pt{
*=0{ \bullet } &&\\
&*=0{ \bullet} \ar@{-}[ul] \ar@{-}[ur] &\\
&*=0{}\ar@{-}[u] &
}
\]
such that the inner face of $x$ is degenerate. Then $x$ is non-degenerate, but the differential
\[
d\left(
\xymatrix@R=7pt@C=5pt{
*=0{ \bullet } &&\\
&*=0{ \bullet} \ar@{-}[ul] \ar@{-}[ur] &\\
&*=0{}\ar@{-}[u] &
}
\right) = \xymatrix@R=7pt@C=7pt{ 
&&\\
&*=0{ \bullet} &\\
&*=0{}\ar@{-}[u] &
}
-
\xymatrix@R=7pt@C=7pt{ 
& *=0{} &\\
&*=0{ \bullet} \ar@{-}[u] &\\
&*=0{}\ar@{-}[u] &
}
+
\xymatrix@R=7pt@C=7pt{ 
&&\\
&*=0{ \bullet} \ar@{-}[ul] \ar@{-}[ur] &\\
&*=0{}\ar@{-}[u] &
}
\]
calculated in $\Ch^{un}(X)$ is not a linear combination of non-degenerate dendrices.
More informally we can describe the differential of $\Ch(X)$ as a modification of the differential of $\Ch^{un}(X)$ where we disregard 
degenerate dendrices. 
\end{remark}

\section{Equivalence of the chain complexes}

\begin{proposition}\label{prop_adjunction}
Let $\Gamma_d^{un}:\Ch_{\geq0}  \to \dSet$ be defined by the formula $$\Gamma_d^{un}(C)_T=\Hom_{\Ch_{\geq0}}(\Ch^{un}(\Omega[T]), C).$$
Then the pair $(\Ch^{un}, \Gamma_d^{un})$ forms an adjunction. 
\end{proposition}

\begin{proof}
It is well known that a functor from a presheaf category to a cocomplete category is left adjoint if and only if it preserves colimits. We want to show that $\Ch$ preserves colimits.
Since colimits in chain complexes are just colimits of the underlying graded abelian groups it suffices to show that for every $n$ the functor
$$ \Ch^{un}(X)_n =  C(X)_n / A(X)_n $$
preserves colimits in $X$. We write $\Ch^{un}$ as a coequalizer
$$
\bigoplus_{\begin{small}
\begin{array}{c}
T,T'\in \Omega, \\
 |T|=|T'|= n
\end{array}
\end{small}} \bigoplus_{\begin{small}\begin{array}{c}
\tau: T \iso T' 
\end{array}\end{small}} 
\bigoplus_{\begin{small}
\begin{array}{c} p\in P_T, \\ p' \in P_{T'} \end{array}
\end{small}} \mathbb{Z}\langle X_T \rangle \, \rightrightarrows \bigoplus_{\begin{small}
\begin{array}{c} T\in \Omega,\\ |T|=n\end{array}
\end{small}  } \bigoplus_{\begin{small}
\begin{array}{c} p\in P_T\end{array}
\end{small}} \mathbb{Z}\langle X_T \rangle 
$$ 
where the maps are given on generators by 
$$
(T,T',\tau,p,p',x) \mapsto (T,p,x) \qquad\text{and}\qquad (T,T',\tau,p,p',x) \mapsto (T',p',  \sgn(p', \tau^* p) \tau^*(x)).
$$
Now we see that both sides of the coequalizer commute with colimits in $X$ since the direct sum functor and the free abelian group functor commute with colimits.
Its also clear that the maps between the two abelian groups commute with colimits since they are (apart from a sign) completely determined by the indexing set.
Since coequalizers also commute with colimits this finishes the proof. 
\end{proof}

\begin{proposition} \label{prop_adjunction_norm}
Let $\Gamma_d:\Ch_{\geq0}  \to \dSet$ be defined by the formula $$\Gamma_d(C)_T=\Hom_{\Ch_{\geq0}}(\Ch(\Omega[T]), C).$$
Then the pair $(\Ch, \Gamma_d)$ forms an adjunction. 
\end{proposition}
\begin{proof}
The proof is similar to the proof of Proposition \ref{prop_adjunction}. We again want to show that the functor
$$ \Ch(X) = \Ch^{un}(X) / D(X)$$
preserves colimits in $X$. 
We consider the following functor
$$
\Xi: \dSet \to \Ab \qquad X \mapsto \bigoplus_{\begin{small}
\begin{array}{c}  |T|= n\end{array}
\end{small}}  
\bigoplus_{\begin{small}
\begin{array}{c} \tau: T \to T' \\ \text{degeneracy}\end{array}
\end{small}}
\bigoplus_{\begin{small}
\begin{array}{c} p\in P_T\end{array}
\end{small}}  \mathbb{Z}\langle X_T' \rangle
$$
Then there is a natural transformation $\Xi \to \Ch^{un}$ given on generators by 
$$ 
(T, \tau, p, x) \mapsto [T,p , \tau^*x]
$$
By definition it is clear that $\Ch(X)$ is the cokernel of $\Xi(X) \to \Ch^{un}(X)$. Thus the fact that everything clearly commutes with colimits shows the claim.
\end{proof}

Recall that the category $\Ch_{\geq0}$ of positively graded chain complexes admits two canonical model structures. The projective one and the injective one. In both the weak equivalences are quasi-isomorphisms. In the injective model structure the cofibrations are all monomorphisms and in the projective one the cofibrations are the monomorphisms with levelwise projective cokernel.

\begin{proposition} \label{gen_cof}
The functor $\Ch: \dSet \to \Ch_{\geq0}$ maps boundary inclusions to cofibrations (in either of the model structures). The same is true for the functor $\Ch^{un}$.
\end{proposition}

\begin{proof} 
Let $i: \partial\Omega[S] \to \Omega[S]$ be a boundary inclusion. Because $\partial\Omega[S]_T  \to \Omega[S]_T$ are monomorphisms compatible with the relation (\ref{relation}), the induced maps $\Ch(i)_n$ and $\Ch^{un}(i)_n$ are monomorphisms between free abelian groups given by inclusion of generators. Hence their cokernels are also free.  
\end{proof}

\begin{corollary}
The natural map $ \Ch^{un}(X) \to \Ch(X)$ is a quasi-isomorphism for every normal dendroidal set $X$. 
\end{corollary}

\begin{proof}
By Lemma \ref{AcycDeg}, $D(\Omega[T])$ is acyclic for every tree $T$. 
Hence, the natural maps 
\[\Ch^{un}(\Omega[T]) \to \Ch(\Omega[T])\] are quasi-isomorphisms.  Proposition \ref{prop_adjunction}, Proposition \ref{prop_adjunction_norm}, Proposition \ref{gen_cof} and Lemma \ref{LemmaEq} imply the result. 
\end{proof}

\begin{proposition} \label{gen_triv_cof}
The functor $\Ch\colon \dSet \to \Ch_{\geq0}$ maps 
horn inclusions $\Lambda^a[T]\to \Omega[T]$ to trivial cofibrations (in either of the model structures). 
The same is true for the functor $\Ch^{un}$.
\end{proposition}

\begin{proof} 
By Proposition \ref{gen_cof}, it is enough to show that the functor $\Ch$ sends a horn inclusion $i: \Lambda^a[T]\to \Omega[T]$ to a quasi-isomorphism. Let $|T|=n$. Then $\Ch(\Lambda^a[T])_k\to \Ch(\Omega[T])_k$ is an isomorphism for $0\leq  k\leq n-2$. Hence, $H_k(i)$ is an isomorphism for $0\leq k \leq n-3$. 

Note that $\Ch(\Lambda^a[T])_{n-1}$ is a subgroup of $\Ch(\Omega[T])_{n-1}$ generated by all but one generator, let us denote it $[x_a]$,  of $\Ch(\Omega[T])_{n-1}$. The group $\Ch(\Lambda^a[T])_n$ is trivial and $\Ch(\Omega[T])_n$ is generated by one element, call it $[x]$. Then $[x_a]-d([x])$ is in $\Ch(\Lambda^a[T])_{n-1}$, so $d([x_a])$ is in $d(\Ch(\Lambda^a[T])_{n-1})$. This implies that $H_{n-2}(i)$ is an isomorphism. 
Also,  the fact that $[x_a]-d([x])$ is an element of $\Ch(\Lambda^a[T])_{n-1}$ implies $H_{n-1}(i)$ and $H_n(i)$ are isomorphisms. 
\end{proof}

From Lemma \ref{sep18}, Proposition \ref{gen_cof} and Proposition \ref{gen_triv_cof} we have the following immediate consequence.

\begin{corollary}
The adjunctions $(\Ch, \Gamma_d)$ and $(\Ch^{un}, \Gamma_d^{un})$ are Quillen adjunctions between the category of dendroidal sets with the stable model structure and the category of non-negatively graded chain complexes with either the projective or the injective model structure.
\end{corollary}

\begin{remark}
The last result in particular implies the following fact: given a chain complex $C$ the classical Dold-Kan correspondence associates to it a simplicial set $\Gamma(C)$ (in fact a simplicial abelain group). 
From our constructions it follows that $\Gamma(C)$ underlies the dendroidal set $\Gamma_d(C)$ (which is in fact a dendroidal abelian group) and that $\Gamma_d(C)$ is fully Kan. This observation can be promoted to a 
dendroidal Dold-Kan correspondence (slightly different in spirit to the one in \cite{GLW} which only works for planar dendroidal sets).
\end{remark}

\begin{definition}
For a dendroidal set $X$ define the homology and cohomology groups with values in an abelian group $A$ as  
$$ H_n(X, A) := H_n(\Ch^{un}\tilde X \otimes A) \qquad \text{and} \qquad H^n(X,A) := H_n( \Hom(\Ch^{un}\tilde X, A)) . $$
where $\tilde X \to X$ is a normal, i.e. cofibrant replacement of $X$. We will write $H_n(X)$ for $H_n(X,\Z).$ 
\end{definition}

\begin{remark}
For a dendroidal set of the form $i_! S$ where $S$ is a simplicial set the chain complex $Ch^{un}(i_! S)$ agrees with the unnormalized chain complex of $S$. Since $i_!S$ is normal, we have 
$$H_n(i_! S,A) \cong H_n(S,A) \qquad \text{and} \qquad H^n(i_! S,A) \cong H^n(S,A)\ .$$
\end{remark}

\begin{corollary}\label{cor_homologyequi}
If $f: X \to Y$ is a stable equivalence of dendroidal sets, then it induces an isomorphism $f_*: H_n(X,A) \to H_n(Y,A)$.
\end{corollary}

Note that we will show in Corollary \ref{equiv} that the converse of that statement is also true.

\begin{corollary}\label{cor_comm}
For the terminal dendroidal set $\ast$ we have $H_k(\ast) = 0$ for all $k$.
\end{corollary}
\begin{proof}
Since the homotopy category of dendroidal sets with respect to the stable model structure is equivalent to the homotopy category of connective spectra, it follows that it is pointed, i.e.
that the initial object is isomorphic to the terminal object. This means that the canonical morphism $\emptyset \to \ast$ is a stable weak equivalence. Thus we conclude that the homology of $\ast$ is isomorphic
to the homology of $\emptyset$ which is clearly zero in all degrees.
\end{proof}

\begin{corollary} \label{cor_eins}
The homology of $\Omega[T]$ is given by 
\begin{equation*}
H_k(\Omega [T])=\left\{ 
\begin{array}{ll} \mathbb{Z}\langle \ell(T) \rangle & \textrm{if } k= 0, \\ 0 & \textrm{if } k \neq 0. \end{array} 
\right.
\end{equation*}
\end{corollary}
\begin{proof} The morphism \[\bigsqcup_{\ell(T)} \eta \to \Omega[T]\] is a stable trivial cofibration, so the result follows from Corollary \ref{cor_homologyequi}. 
\end{proof}

\begin{corollary} \label{cor_zwei}

 Let $T$ be a tree with $n$ vertices. Then we have
\begin{equation*}
H_k(\Omega [T]/\partial\Omega[T])=\left\{ 
\begin{array}{ll} \Z, & \textrm{if } k=n \\ 0, &  \textrm{if } k \neq n \end{array} 
\right.
\end{equation*}
\end{corollary}
\begin{proof}
We first consider the following pushout square
$$
\xymatrix{
\partial \Omega[T] \ar[d] \ar[r] & \Omega[T] \ar[d] \\
\ast \ar[r] &  \Omega[T] / \partial \Omega[T]
}
$$ 
This square is a homotopy pushout square which can be seen as follows: take the product of the whole square with a cofibrant resolution of $\ast$. Then we get another square in which all corners are
cofibrant and which is a pushout since $\dSet$ is Cartesian closed. This new square is a homotopy pushout since the upper horizontal morphism is a cofibration. But all corners are equivalent to the corners in the starting square, 
this shows that the starting square is also a homotopy pushout square. It follows that we have a homotopy pushout square of chain complexes
$$
\xymatrix{
\Ch( \partial \Omega[T]) \ar[d] \ar[r] & \Ch( \Omega[T]) \ar[d] \\
\Ch(\tilde\ast) \ar[r] &  \Ch(\widetilde{\Omega[T] / \partial \Omega[T]})
}
$$ 
Now we have that $\Ch(\tilde \ast)$ is quasi isomorphic to the zero chain complex by Corollary \ref{cor_comm}. Thus we find that $\Ch(\widetilde{\Omega[T] / \partial \Omega[T]})$ is quasi-isomorphic to the 
homotopy cofibre of the morphism $\Ch( \partial \Omega[T]) \to\Ch (\Omega[T])$. Since this morphism is a monomorphism of chain complexes, the homotopy cofibre is quasi-isomorphic to the quotient. 

This quotient as a chain complex is completely concentrated in degree $n$, since the non-degenerate cells of $\partial \Omega[T]$ and $\Omega[T]$ agree in all other degrees.
\end{proof}

\section{The associated spectrum and its homology}

In this section we will compare the homology of a dendroidal set to the homology of the associated connective spectrum. 
Recall that for a spectrum $E$, its $n$-th homology group with coefficients in an abelian group $A$ is defined as the $n$-th homotopy group of the spectrum $E\wedge HA$, where $HA$ is the Eilenberg-Maclane spectrum of $A$.  
The cohomology groups of $E$ are defined as the negative homotopy groups of the mapping spectrum $HA^E$.

\begin{thm}\label{mainthm}
Let $D$ be a dendroidal set. Then the homology groups $H_*(D,A)$ are naturally isomorphic to the homology groups with values in $A$ of the associated connective spectrum $\mathcal{K}(D)$. The cohomology groups $H^*(D,A)$ are isomorphic to the 
cohomology groups of $\mathcal{K}(D)$.
\end{thm}
\begin{proof}

We consider the following diagram of $\infty$-categories 
\[
\xymatrix{\dSet_{\infty} \ar[r] \ar[d]^{\Ch} & \Sp^{\geq 0} \ar[d]_{- \wedge H\Z}\\ (\Ch_{\geq 0})_\infty \ar[r] & \Mod(H\Z)^{\geq 0}} 
\]
which is a priori not necessarily commutative.
Here $\Sp^{\geq 0}$ denotes the $\infty$-category of connective spectra and $\Mod(H\Z)^{\geq 0}$ is the $\infty$-category of module spectra in $\Sp_{\geq 0}$ over the ring spectrum $H\Z$. The categories on the left side are the underlying $\infty$-categories of the stable model category of dendroidal sets and the category of positive chain complexes. 
The top row is an equivalence of $\infty$-categories as a consequence of Theorem \ref{thm_equivalence}. The bottom row is an equivalence of $\infty$-categories given by the extension of the Dold-Kan correspondence to spectra, Theorem 5.1.6. in \cite{SchwedeShipley} or by the fact that $\Mod(H\Z)^{\geq 0}$ has $H\Z$ as a compact generator. The left vertical map is induced by the left Quillen functor $\Ch$ studied in the previous sections. The right vertical map is given by taking the homology of a spectrum, i.e. by the smash product with $H\Z$.  

The $\infty$-category $\Mod(H\Z)^{\geq 0}$ is an additive $\infty$-category (see Definition 2.6 in \cite{GGN}). The dendroidal set $\eta$ corresponds to the sphere spectrum and its homology is just the spectrum $H\Z$ (as the sphere spectrum is the unit for the smash product). On the other hand the chain complex $\Ch(\eta)$ is just $\Z$ concentrated in degree $0$ and under Dold-Kan it corresponds to $H\Z$. 

Hence there are two left adjoint $\infty$-functors from $\dSet_{\infty}$ to $\Mod(H\Z)^{\geq 0}$ and since they coincide on $\eta$ the Proposition \ref{prop_cool} implies that these functors are equivalent.  
This proves the case of the homology with $\Z$-coefficients. The other cases follow from that.

\end{proof}

\begin{corollary} \label{equiv}
A morphism $f: X \to Y$ between dendroidal sets is a stable weak equivalence if and only if it is a homology isomorphism, i.e. $f_*: H_n(X) \to H_n(Y)$ is an isomorphism for each $n$.
\end{corollary}
\begin{proof}
This follows immediately since it holds for connective spectra which can be seen using Hurewicz's theorem. 
\end{proof}

\begin{corollary}\label{cor_spheres}
The spectrum associated to the dendroidal set $\Omega[T] / \partial \Omega[T]$ is equivalent to the $n$-sphere, i.e. $\Sigma^n \mathbb{S} \simeq \Sigma^\infty (S^n,*)$.
\end{corollary}
\begin{proof}
The only spectrum $E$ such that $H_n(E) = \mathbb{Z}$ and $H_k(E) = 0$ for $k \neq n$ is $\Sigma^\infty S^n$.
\end{proof}

\begin{remark}
The last corollary has the following consequence. Let $X$ be a normal dendroidal set. We can consider the skeletal filtration 
$$
X_0 \subset X_1 \subset X_2 \subset .... \qquad\qquad \bigcup X_n = X
$$
as discussed in \cite{MoerBar}. The subquotients $X_n/ X_{n-1}$ are unions of dendroidal sets $\Omega[T] / \partial \Omega[T]$ where $T$ has $n$ vertices. After passing to the associated spectra this induces a filtration
$$
\calK(X_0) \to \calK(X_1) \to \calK(X_2) \to  .... \qquad\qquad \underrightarrow{\mathrm{lim}} \calK(X_n) \simeq \calK(X)
$$
whose subquotients $\calK{X_n}/ \calK{X_{n-1}}$ are wedges of $n$-spheres by corollary \ref{cor_spheres}. Thus it has to agree with the stable cell filtration of the spectrum $\calK(X)$. The associated spectral sequence thus is the 
Atiyah-Hirzebruch spectral sequence. It was our initial hope that the skeletal filtration of dendroidal sets would lead to more interesting filtrations of $K$-theory spectra.
\end{remark}

Let $A_\infty = N_d(Ass)$ be the dendroidal nerve of the operad for associative algebras. Note that $A_\infty$ is the presheaf of planar structures, which we earlier denoted by $P$. 

\begin{thm}\label{Ainfinity}
The homology of $A_\infty$ vanishes. Therefore the spectrum $\calk(A_\infty)$ is trivial. 
\end{thm}

\begin{proof}
By definition, the generators of the free abelian group $Ch^{un}(A_\infty)_n$ are in bijection with the isomorphism classes of planar structures of trees with $n$ vertices. More precisely, for each tree $T$ there is exactly one generator for each orbit of the action of the group $\Aut(T)$ on the set of planar structures of $T$. Hence we may represent the generators by planar trees with all the edges of the same colour, keeping in mind that isomorphic planar trees are identified. 

For example, for each of the following two shapes the two planar structures get identified, so there is only one generator: 
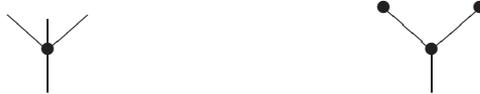
\[
\xymatrix@R=10pt@C=12pt{
&&&&&&&\\
&&&&&&&\\
&&&& *=0{\bullet} \ar@{-}[ul] \ar@{-}[u] \ar@{-}[ur]  &&& \\
&&&&*=0{}\ar@{-}[u] &&&
}
\quad 
\xymatrix@R=10pt@C=12pt{
&&&&&&&\\
&&& *=0{\bullet} && *=0{\bullet} &&\\
&&&& *=0{\bullet} \ar@{-}[ul] \ar@{-}[ur]  &&& \\
&&&&*=0{}\ar@{-}[u] &&&
}
\]
but the following two planar trees are representing two different generators:  
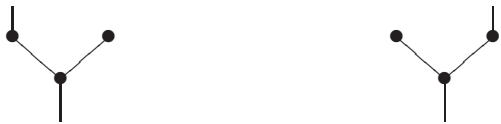
\[
\xymatrix@R=10pt@C=12pt{
&&&&&&&\\
&&& *=0{\bullet} \ar@{-}[u] && *=0{\bullet}  &&\\
&&&& *=0{\bullet} \ar@{-}[ul] \ar@{-}[ur]  &&& \\
&&&&*=0{}\ar@{-}[u] &&&
}
\quad 
\xymatrix@R=10pt@C=12pt{
&&&&&&&\\
&&& *=0{\bullet} && *=0{\bullet} \ar@{-}[u] &&\\
&&&& *=0{\bullet} \ar@{-}[ul] \ar@{-}[ur]  &&& \\
&&&&*=0{}\ar@{-}[u] &&&
}
\]
We call a generator \emph{canonical} if the leftmost top vertex of such a representative is a stump. For example, in the above pictures, the planar trees on the right represent canonical generators, while the ones on the left represent non-canonical generators.

Let $A_n$ (i.e. $B_n$) be the set of canonical (i.e. non-canonical) generators of $\Ch^{un}(A_\infty)_n$.
A bijection between $x\in B_n$ and $\hat{x}\in A_{n+1}$ is obtained by putting a stump on the leftmost leaf of the chosen representative of a non-canonical generator in $B_n$. 

Obviously, a dendrex with no vertices has no stumps, so $A_0$ is empty. The set $B_0$ is a singleton, consisting of the tree with one edge. Also $A_1$ is a singleton containing just the stump. 
For every generator $x$ we define its weight $w(x)$ as the number of leaves of the planar tree representing it if $x$ is non-canonical and $w(x)=0$ if $x$ is canonical. 

If $x$ is non-canonical, then $\hat{x}$ has exactly one leaf less than $x$. Every other face of $\hat{x}$ is either canonical (containing the added stump) or it is a non-canonical face obtained by contracting the edge just below the added tree, so it has one leaf less than $x$.  This shows that all the assumptions of Proposition \ref{acyclicity} hold. Hence all homology groups of $A_\infty$ vanish. 
\end{proof}

\section{Acyclicity argument}

In this section we finally prove the technical proposition which we have used in Lemma \ref{AcycDeg} and Theorem \ref{Ainfinity}  to show acyclicity of certain chain complexes. 

\begin{proposition} \label{acyclicity}
Let $C_\bullet$ be a chain complex such that all $C_n$ are free abelian groups which have a grading 
\[
C_n = \bigoplus_{i\in \mathbb{N}_0} C_{n,i}.
\] 
For $\displaystyle x\in \bigoplus_{i=0}^ m C_{n,i}\setminus \bigoplus_{i=0}^{m-1}C_{n,i}$ we write $w(x)=m$. Let $A_n$ and $B_n$ be a basis for $C_{n,0}$ and $\displaystyle \bigoplus_{i>0} C_{n,i}$, respectively. Assume there is a bijection between the sets $B_n$ and $A_{n+1}$ which sends $x\in B_n$ to $\hat{x}\in A_{n+1}$ and one of the following two statements holds
\begin{equation*}
w(x-d(\hat{x})) < w(x) \quad \textrm{ or } \quad w(x+d(\hat{x})) < w(x).
\end{equation*}
Then $H_0(C_\bullet)=\mathbb{Z} \langle A_0 \rangle$ and $H_n(C_\bullet)=0$ for all $n\geq 1$. 
\end{proposition}

\begin{proof}
First, for each $x\in B_n$ we construct an element $\bar{x}\in C_{n+1,0}$ such that 
\[
x-d(\bar{x})\in C_{n,0}. 
\]

We proceed by induction on $w(x)$. If $w(x)=1$, we can take $\bar{x}$ to be $\hat{x}$ or $-\hat{x}$ and the statement follows by assumption. 
Let $w(x)>1$ and assume that the statement holds for all $y\in B_n$ such that $w(y)<w(x)$. We let $x'=\pm\hat{x}$, where the sign $\pm$ is such that $w(x-d(x'))<w(x)$. We write
\[
x-d(x')=z+y
\] 
where $z\in C_{n,0}$, $y\in C_n\setminus C_{n,0}$, and $y$ is a finite sum of elements $y_i$ in $B_n$ such that $w(y_i)<w(x)$ for $i=1,\ldots, k$. By the inductive hypothesis, we have $\bar{y_i}\in C_{n+1,0}$ such that $y_i-d(\bar{y_i})\in C_{n,0}$, for $i=1, \ldots, k$.  Our claim now follows if we let $\bar{x}=x'+\sum_i \bar{y_i}$.

Note that this same inductive argument shows that every element $\hat{x}$, $x\in B_n$, can be written as a linear combination of elements of the set 
$\{\bar{x} : x\in B_n\}$. As we assumed $A_{n+1}=\{\hat{x} : x\in B_n\}$ is a basis for $C_{n+1,0}$, it follows that the set $\{\bar{x} : x \in B_n\}$ generates $C_{n+1,0}$. 

We will show that the set $\{d(\bar{x}) : x \in B_n\}$ is linearly independent, for every $n$. 
Let  us assume $\sum_{i=1}^k \alpha_i d(\bar{x_i})=0$ for some $x_1,...,x_k\in B_n$.  We can write $d(\bar{x_i})=x_i+y_i$, where $y_i\in C_{n,0}$ for $i=1,2,...,k$. Hence we have 
\[
\sum_{i=1}^k \alpha_i x_i+\sum_{i=1}^k \alpha_i y_i=0.
\]
We conclude that $\alpha_i=0$ for all $i$ since $y_1,\ldots, y_k\in C_{n,0}$, $x_1,\ldots, x_k\in B_n$ and $B_n$ is a  basis for $C_n\setminus C_{n,0}$. Since $d$ is linear, the set $\{\bar{x} : x \in B_n\}$ is also linearly independent, for every $n$. This implies that the set $\{\bar{x} : x \in B_n\}$ is a basis for $C_{n+1,0}$. 

Next we show that the restriction $d\colon C_{n,0} \to \img d$ is surjective. 
Let $y=d(a+b)$ be an element of $\img d$ with $a\in C_{n,0}$ and $b\in C_{n}\setminus C_{n,0}$. There is an element $\bar{b}\in C_{n+1,0}$ such that $b-d\bar{b} \in C_{n,0}$. 
Since $d^2=0$ we have $y=d(a+b)=d(a+b)-d(d(\bar{b}))=d(a+b-d\bar{b})\in d(C_{n,0})$. 
It follows that $\{d(\bar{x}) : x\in B_n\}$ is a basis for $\img d$. We conclude that the restriction 
\[
d\colon C_{n+1, 0}  \to \img d=\text{span}\{d(\bar{x})\}
\]
is an isomorphism for every $n$. 

Furthermore, this implies that $\Ker d$ is disjoint with $C_{n,0}$ for every $n$.  As $d(\bar{x})\in \Ker d$, the set $\{d(\bar{x}) : x \in B_n\}$ is also disjoint with $C_{n,0}$ and by the construction of $\bar{x}$ we have that $\text{span}\{d(\bar{x}) \} \oplus C_{n,0} = C_n$. 
We also have $\Ker d \oplus C_{n,0}=C_n$ because $C_{n,0} \to \img d$ is an isomorphism. Since $\text{span}\{d(\bar{x}) \}\subseteq \Ker d$, we must have \[ \Ker d = \text{span}\{d(\bar{x}) \} = \img d, \] so $H_n(C_\bullet)=0$ for all $n\geq 1$. 
\end{proof}

\bibliographystyle{amsalpha}
\bibliography{Alles}

\end{document}